\documentclass{amsart}
\usepackage{amsfonts}
\usepackage{amsfonts}
\usepackage{amsfonts}
\usepackage{amsfonts}
\usepackage{amsfonts}
\usepackage{bbm}
\usepackage{}

\usepackage{amsfonts}
\usepackage{amsfonts,amssymb,amsmath,amsthm}
\usepackage{mathrsfs}
\usepackage{graphicx}
\usepackage{url}
\usepackage{enumerate}

\newtheorem{theorem}{Theorem}[section]
\newtheorem{lemma}[theorem]{Lemma}
\newtheorem{proposition}[theorem]{Proposition}
\newtheorem{corollary}[theorem]{Corollary}
\theoremstyle{definition}
\newtheorem{definition}[theorem]{Definition}
\newtheorem{remark}{Remark}
\newtheorem{example}{Example}

 \DeclareMathOperator{\Cov}{Cov}
\DeclareMathOperator{\diam}{diam} 
\DeclareMathOperator{\vo}{Vol}\DeclareMathOperator{\id}{id}
\DeclareMathOperator{\Exp}{Exp}
\DeclareMathOperator{\LGC}{LGC}

\DeclareMathOperator{\A}{A}

\author{Lixia Yuan}
\address{
 Department of Mathematics\\
  Fudan University\\
  Shanghai, China}
\email{yuan\underline{ }lixia@foxmail.com}
\author{Wei Zhao}
\address{
Department of Mathematics\\
East China University of Science and Technology\\
Shanghai, China}
\email{szhao\underline{ }wei@yahoo.com}

\begin{document}

\title[]{Some formulas of Santal\'o type in Finsler geometry and its applications}

\begin{abstract}
In this paper, we establish two Santal\'o type formulas for general Finsler manifolds. As applications, we derive a universal lower bound for the first eigenvalue of the nonlinear Laplacian, two Croke type isoperimetric inequalities, and a Yamaguch type finiteness theorem in Finser geometry.
\end{abstract}
\maketitle

\section{Introduction}
In \cite{Sa,Sa2}, Santal\'o considered the kinematic measure and established a formula which describes the Liouville measure on the unit sphere bundle of a Riemannian manifold in terms of the geodesic flow and the measure of a hypersurface.
This formula plays an important role in global Riemannian geometry.
Some of its applications are universal bounds for the first eigenvalue \cite{C2}, Croke's isoperimetric inequality \cite{Cr} and a generalization of Berger's theorem \cite{Cr3}. Moreover, with Santal\'o's formula, Croke in \cite{Cr2} solved a famous isoperimetric problem in dimension $4$. See \cite{C2,Cr2,Cr3,Cr5,Cr,Cr4,Sa,Sa2} for more details.

 A Finsler manifold is a differentiable manifold, on
which every tangent space is endowed a Minkowski norm instead of a
Euclidean norm.
There is only one reasonable notion of the measure for
Riemannian manifolds (cf. \cite{BBI}). However, the measure on a Finsler manifold can be defined in various ways and essentially different
results may be obtained, e.g., \cite{AlB,AlT,Sh1}. Hence,
it is interesting
to ask whether an analogue of Santal\'o's formula still holds for Finsler manifolds.

Let $(M,F)$ be a Finsler manifold. Denote by $\pi:SM\rightarrow M$ the unit sphere bundle. If $F(y)=F(-y)$ for any $y\in SM$, then $F$ is reversible. In a reversible Finsler manifold, the reverse of a geodesic is still a geodesic (see \cite{BCS,Sh1}).
In \cite{ZY}, Shen and Zhao considered the problem above and established a Santal\'o type formula for reversible Finsler manifolds.

There are infinitely many nonreversible Finsler metrics.
For example, a Randers metric in the form $F=\alpha+\beta$ is non-reversible, where
$\alpha$ is a Riemannian metric and $\beta$ is a $1$-form. The reverse of a geodesic in a non-reversible Finsler manifold is in general not a geodesic.
Moreover, in a non-reversible Finsler manifold, the measure of a hypersurface induced by the inward normal vector field may be different from the one induced by the outward normal vector filed (see Example 1 in Section 5 below). The purpose of this paper is to establish some Santal\'o type formulas for general Finsler manifolds.

Let $(M,\partial M, F,d\mu)$ be a compact Finsler manifold with smooth boundary, where $F$ is possibly non-reversible and $d\mu$ is a measure on $M$.
 Denote by $\mathbf{n}_+$ and $\mathbf{n}_-$ the unit inward and outward normal vector fields along $\partial M$, respectively. The measures on $\partial M$ induced by $\mathbf{n}_\pm$ are defined by $d\A_\pm:=i^*(\mathbf{n}_\pm\rfloor d\mu)$.
Let $S^+\partial M$ and $S^-\partial M$ be the bundles of inwardly and outwardly pointing unit vectors along $\partial M$, i.e., $S^\pm\partial M=\{y\in
SM|_{\partial M}:g_{\mathbf{n}_\pm}(\mathbf{n}_\pm,y)>0\}$. The measures on $S^\pm\partial M$ are the product measures $d\chi_\pm(y):=d\nu_{\pi(y)}(y)d\A_\pm(\pi(y))$, where $d\nu_x(y)$ is the Riemannian measure on $S_xM:=\pi^{-1}(x)$ induced by $F$.
For each $y\in S^+\partial M$,
set ${\mathfrak{t}}(y):=\sup \{t>0:\gamma_y(s)\in M, \,0<s<t\}$ and $l(y):=\min \{i(y),\mathfrak{t}(y)\}$, where $i(y)$ is the cut value of $y$.

Since $F$ may be non-reversible, to investigate the asymmetry of the Finsler manifold, we introduce the reverse of $F$, which is defined by $\widetilde{F}(y):=F(-y)$. Clearly, $\widetilde{F}$ is a Finsler metric as well.  Let $\tilde{\mathfrak{t}}(\cdot)$, $\tilde{i}(\cdot)$ and $\tilde{l}(\cdot)$ be defined as above on $(M,\partial M,\widetilde{F})$.
Then we have the following Santal\'o type formulas.
\begin{theorem}\label{t1}
 For all integral function $f$ on $S M$, we have
\begin{align*}
&\int_{\mathcal {V}^-_M}f dV_{SM}=\int_{y\in S^+\partial M} e^{\tau(y)}g_{\mathbf{n}_+}(\mathbf{n}_+,y)d\chi_+(y)\int^{l(y)}_0f(\varphi_t(y))dt,\tag{i}\label{th1i}\\
&\int_{\mathcal {V}^+_M}f dV_{SM}=\int_{y\in S^-\partial M} e^{\tau(y)}g_{\mathbf{n}_-}(\mathbf{n}_-,y)d\chi_-(y)\int_0^{\tilde{l}(-y)}f(\varphi_{-t}(y))dt,\tag{ii}\label{th1ii}
\end{align*}
where $dV_{SM}$ is the canonical Riemannian measure on $SM$, $\tau$ is the distortion of $d\mu$, $\varphi_t(y)$ is the geodesic flow of $F$, $\mathcal {V}^-_M:=\{y\in S M: \tilde{\mathfrak{t}}(-y)\leq  \tilde{{i}}(-y)\}$ and $\mathcal {V}^+_M:=\{y\in S M: {\mathfrak{t}}(y)\leq  {{i}}(y)\}$.
\end{theorem}
One can easily see that Theorem \ref{t1} implies the Santal\'o type formulas for reversible Finsler manifolds \cite{ZY} and for Riemannian manifolds \cite{Sa,Sa2}. It is remarkable that, in a non-reversible Finsler manifold, $\A_-(\partial M)\neq \A_+(\partial M)$ and the formulas (\ref{th1i}) and (\ref{th1ii}) contain information about $\A_+(\partial M)$ and $\A_-(\partial M)$, respectively.

Before giving some applications of Theorem \ref{t1}, we shall
recall some notions and basic facts of the first eigenvalue in the Finsler setting.
The first eigenvalue $\lambda_1(M,d\mu)$ in $(M,F,d\mu)$ is defined as
the smallest positive eigenvalue of the nonlinear Laplacian $\Delta_{d\mu}$ introduced by Shen (cf. \cite{GS,Sh1,Sh3}).
It should be noted that both $\Delta_{d\mu}$ and $\lambda_1(M,d\mu)$ are dependent on the choice of the measure $d\mu$.
Theorem \ref{t1} now yields the following
\begin{theorem}\label{t2}
Let $(M,\partial M,F)$ be a compact Finsler $n$-manifold with smooth boundary such that every geodesic ray in $(M,{F})$ minimizes distance up to the point that it intersects $\partial M$.  Then
\[
\lambda_1(M,d\mu)\geq\left \{
\begin{array}{lll}
&\frac{\lambda_D(\mathbb{S}_D^+)}{\Lambda_F^{4n+1}},& d\mu\text{ is the Busemann-Hausdorff measure},\\
\\
&\frac{\lambda_D(\mathbb{S}_D^+)}{\Lambda_F^{2n+1}},& d\mu\text{ is the Holmes-Thompson measure},
\end{array}
\right.
\]
where $D:=\diam(M)$, $\Lambda_F$ is the uniform constant of $F$, and $\mathbb{S}_D^+$ denotes the $n$-dimensional Riemannian hemisphere of the constant sectional curvature sphere having diameter equal to $D$. The equality holds if and only if $(M,F)$ is isometric to $\mathbb{S}_D^+$.
\end{theorem}
Note that a Finsler metric $F$ is Riemannian if and only if $\Lambda_F=1$. Hence, Theorem \label{t2} implies Croke's sharp universal lower bound for the first eigenvalue \cite{C2,Cr}.

Let $(M,\partial M,F)$ be as before.
For each $x\in M$, set
\[
\omega:=\inf_{x\in M}\min\{\omega^+_x,\,\omega^-_x\},
\]
where $\omega_x^\pm:=c_{n-1}^{-1}\int_{U_x^\pm}e^{\tau(y)}d\nu_x(y)$, $U_x^\pm:=\pi|_{\mathcal {V}^\pm_M}^{-1}(x)$ and $c_{n-1}=\vo(\mathbb{S}^{n-1})$.
Then Theorem \ref{t1} furnishes the following inequalities.

\begin{theorem}\label{t3}
Let $(M,\partial M, F,d\mu)$ be a compact Finsler $n$-manifold with smooth boundary, where $d\mu$ is either the Busemann-Hausdorff measure or the Holmes-Thompson measure.
Then

(1)
\[
\frac{\A_\pm(\partial M)}{\mu(M)}\geq \frac{(n-1)c_{n-1}\,\omega}{c_{n-2}\, D\,\Lambda_F^{2n+\frac12}},
\]
where $D:=\diam(M)$.

(2)
\[
\frac{\A_\pm(\partial M)}{\mu(M)^{1-\frac1n}}\geq \frac{c_{n-1}\omega^{1+\frac1n}}{(c_n/2)^{1-\frac1n}\Lambda_F^{2n+\frac52}},
\]
with equality if and only if $(M, F)$ is a Riemannian hemisphere of a constant sectional curvature sphere.
\end{theorem}
If $F$ is reversible, then $\omega_+=\omega_-$ and $\A_+(\partial M)=\A_-(\partial M)$. Hence, Theorem \ref{t3} implies Croke type isoperimetric inequalities for reversible Finsler manifolds \cite[Theorem 1.6]{ZY} and for Riemannian manifolds \cite{Cr}.

As an application of Theorem \ref{t3}, we  obtain a Finslerian version of Yamaguchi's finiteness theorem.
\begin{theorem}\label{tlast}
For any $n$ and positive numbers $i, V, \delta$, the class of closed
Finsler $n$-manifolds $(M,F)$ with injectivity radius
$\mathfrak{i}_M\geq i$, $\Lambda_F\leq \delta$ and $\mu(M) \leq V$,
contains at most finitely many homotopy types. Here, $\mu(M)$ is either the Busemann-Hausdorff volume or the Holmes-Thompson volume of $M$.
\end{theorem}

\section{Preliminaries}
In this section, we recall some definitions and properties about Finsler manifolds. See \cite{BCS,Sh1} for more details.

Let $(M,F)$ be a (connected) Finsler $n$-manifold with Finsler metric $%
F:TM\rightarrow [0,\infty)$.
Let $(x,y)=(x^i,y^i)$ be local coordinates on $%
TM$. Define
\begin{align*}
g_{ij}(x,y):=\frac12\frac{\partial^2 F^2(x,y)}{\partial y^i\partial y^j}, \ G^i(y):=\frac14 g^{il}(y)\left\{2\frac{\partial g_{jl}}{\partial x^k}(y)-\frac{\partial g_{jk}}{\partial x^l}(y)\right\}y^jy^k,
\end{align*}
where $G^i$ are the geodesic coefficients. A smooth curve $\gamma(t)$ in $M$ is called a (constant speed) geodesic if it satisfies
\[
\frac{d^2\gamma^i}{dt^2}+2G^i\left(\frac{d\gamma}{dt}\right)=0.
\]
We always use $\gamma_y(t)$ to denote the geodesic with $\dot{\gamma}_y(0)=y$.

The Ricci curvature is defined by
$\mathbf{Ric}(y):=\overset{n}{\underset{i=1}{\sum}}R^i_{\,i}(y)$,
where
\[
R^i_{\,k}(y):=2\frac{\partial G^i}{\partial x^k}-y^j\frac{\partial^2G^i}{\partial x^j\partial y^k}+2G^j\frac{\partial^2 G^i}{\partial y^j \partial y^k}-\frac{\partial G^i}{\partial y^j}\frac{\partial G^j}{\partial y^k}.
\]

Let $\pi:SM\rightarrow M$ be the unit sphere bundle, i.e.,
$S_xM:=\{y\in T_xM:F(x,y)=1\}$ and $SM:=\cup_{x\in M}S_xM$.
The measure on $SM$ is defined by
\begin{align*}
dV_{SM}|_{(x,y)}
&=\sqrt{\det g_{ij}(x,y)} dx^1\wedge \cdots \wedge
dx^n\wedge  d\nu_x(y) \\
&= e^{\tau(y)}\pi^*(d\mu(x))\wedge  d\nu_x(y) .
\end{align*}
where
\[
d\nu_x(y):= \sqrt{\det
g_{ij}(x,y)}\left(\overset{n}{\underset{i=1}{\sum}}(-1)^{i-1}
y^idy^1\wedge\cdots \wedge\widehat{dy^i}\wedge \cdots\wedge dy^n\right).
\]
is the Riemannian measure on $S_xM$ induced by $F$.

 The reversibility $\lambda_F$ and the uniformity constant $\Lambda_F$ of $(M,F)$ are defined by $\lambda_F:=\sup_{x\in M}\lambda_F(x)$ and $\Lambda_F:=\sup_{x\in M}\Lambda_F(x)$, where
\[
\lambda_F(x):=\underset{y\in S_xM}{\sup}F(x,-y),\ \Lambda_F(x):=\underset{X,Y,Z\in S_xM}{\sup}\frac{g_X(Y,Y)}{g_Z(Y,Y)}.
\]
Clearly, ${\Lambda_F}\geq \lambda_F^2\geq 1$. $\lambda_F=1$ if and only if $F$ is reversible, while $\Lambda_F=1$ if and
only if $F$ is Riemannian.

The dual Finsler metric $F^*$ on $M$ is
defined by
\begin{equation*}
F^*(\eta):=\underset{X\in T_xM\backslash 0}{\sup}\frac{\eta(X)}{F(X)}, \ \
\forall \eta\in T_x^*M.
\end{equation*}
The Legendre transformation $\mathfrak{L} : TM \rightarrow T^*M$ is defined
as
\begin{equation*}
\mathfrak{L}(X):=\left \{
\begin{array}{lll}
& g_X(X,\cdot) & \ \ \ X\neq0, \\
& 0 & \ \ \ X=0.%
\end{array}
\right.
\end{equation*}
In particular, $F^*(\mathfrak{L}(X))=F(X)$. Now let $f : M \rightarrow \mathbb{R}$ be a smooth function on $M$. The
gradient of $f$ is defined by $\nabla f = \mathfrak{L}^{-1}(df)$. Thus, $df(X) = g_{\nabla f} (\nabla f,X)$.

Let $d\mu$ be a measure on $M$. In a local coordinate system $(x^i)$,
express $d\mu=\sigma(x)dx^1\wedge\cdots\wedge dx^n$. In particular,
the Busemann-Hausdorff measure $d\mu_{BH}$ and the Holmes-Thompson measure $d\mu_{HT}$ are defined by
\begin{align*}
&d\mu_{BH}=\sigma_{BH}(x)dx:=\frac{\vo(\mathbb{B}^{n})}{\vo(\{y\in T_xM: F(x,y)<1\})}dx^1\wedge\cdots\wedge dx^n,\\ &d\mu_{HT}=\sigma_{HT}(x)dx:=\left(\frac1{c_{n-1}}\int_{S_xM}\sqrt{\det g_{ij}(x,y)}d\nu_x(y) \right) dx^1\wedge\cdots\wedge dx^n.
\end{align*}
For $y\in T_xM\backslash0$, define the distorsion of $(M,F,d\mu)$ as
\begin{equation*}
\tau(y):=\log \frac{\sqrt{\det g_{ij}(x,y)}}{\sigma(x)}.
\end{equation*}
By the same argument as in \cite{W}, one can show the following lemma.
\begin{lemma}
Let $(M,F)$ be a Finsler $n$-manifold with finite uniform constant $\Lambda_F$. Let $d\mu$ denote either the Busemann-Hausdorff measure or the Holmes-Thompson measure on $M$. Then the distortion $\tau$ of $d\mu$ satisfy
$\Lambda_F^{-n}\leq e^{\tau(y)}\leq \Lambda_F^n$,
for all $y\in SM$.
\end{lemma}

The reverse of a Finsler metric $F$ is defined by $\widetilde{F}(y):=F(-y)$. It is not hard to see that
$\widetilde{G}^i(y)=G^i(-y)$ and $d\tilde{\mu}=d\mu,$
where $\widetilde{G}^i$ (resp. $G^i$) are  the geodesic coefficients of $\widetilde{F}$ (resp. $F$), and $d\tilde{\mu}$ (resp. $d\mu$) denotes the Busemann-Hausdorff
measure or the Holmes-Thompson measure of $\widetilde{F}$ (resp. $F$). In particular, if $\gamma$ is a
geodesic of $F$, then the reverse of $\gamma$ is a geodesic of
$\widetilde{F}$.

\section{Santal\'o type formulas}
Let $(M,\partial M,F)$ be compact Finsler manifold with smooth boundary.
Denote by $\mathbf{n}_+$ (resp. $\mathbf{n}_-$) the unit inward (resp. outward) normal vector field along $\partial M$.
Define
$\mathcal {N}_+:=\{k\cdot \mathbf{n}_+(x): x\in \partial M,\, k\in \mathbb{R}\}$. The exponential map $\Exp_+$ of $\mathcal {N}_+$ is defined by
\[
\Exp_+: \mathcal {N}_+\rightarrow M,\, k\cdot\mathbf{n}_+(x)\mapsto \exp_x(k\mathbf{n}_+(x)).
\]
We always identify $\partial M$ with the zero section of $\mathcal {N}_+$.
The same arguments as in \cite[Lemma 5.1, Remark 5.1]{ZY} show the following lemma.
\begin{lemma}\label{Fermi}
$\Exp_+$ maps a neighborhood of $\partial M\subset \mathcal {N}_+$ $C^1$-diffeomorphically onto a neighborhood of $\partial M\subset \overline{M}$. Hence, there exists a small $\delta>0$ such that $\Exp_+: M_\delta\rightarrow \Exp_+(M_\delta)$ is $C^1$-diffeomorphic, where $M_\delta:=\{k\cdot\mathbf{n}_+(x):\, x\in \partial M, \, 0\leq k<\delta\}$.
\end{lemma}

Define $\rho:\overline{M}\rightarrow \mathbb{R}_+$ by $\rho(x)=d(\partial M,x)$. Lemma \ref{Fermi} together with the proofs of \cite[Lemma 5.2-5.3, Corollay 5.1]{ZY} and \cite[Lemma 3.2.3]{Sh1} yields
\begin{lemma}\label{in-in}
Let $\sigma(t)$, $0\leq t<\epsilon$, be a $C^1$-curve with $\sigma(0)\in \partial M$ and $\sigma((0,\epsilon))\subset M$. Then
\[
0\leq \left.\frac{d}{dt}\right|_{t=0^+}\rho\circ\sigma(t)=g_{\mathbf{n}_+}(\mathbf{n}_+,\dot{\sigma}(0)).
\]
Hence, $\nabla\rho_+(x)=\mathbf{n}_+(x)$, for any $x\in \partial M$.
\end{lemma}

Set
$S^\pm\partial M:=\{y\in SM|_{\partial M}: g_{\mathbf{n}_\pm}(\mathbf{n}_\pm,y)>0\}$.
By the Legendre transformations, one can show that $S^\pm\partial M$ are two submanifolds of $S\overline{M}$.
\begin{remark}In general, $\mathbf{n}_+\neq-\mathbf{n}_-$. However, it follows from the Legendre transformations that
$S^\pm\partial M=\{y\in SM|_{\partial M}: g_{\mathbf{n}_\mp}(\mathbf{n}_\mp,y)<0\}$.

\end{remark}

Set $\mathcal {Z}:=\{y\in S\partial M: \exists \, t>0 \text{ such that }\gamma_y((0,t))\subset M\}$. Define a function
${\mathfrak{t}}: SM\cup S^+\partial M\cup\mathcal {Z}\rightarrow \mathbb{R}_+$
by ${\mathfrak{t}}(y):=\sup \{t>0:\gamma_y(s)\in M, \,0<s<t\}$, which is called the ${\mathfrak{t}}$-function. By the same argument as in \cite[Lemma 5.4]{ZY}, one can show that ${\mathfrak{t}}$-function is low semi-continuous on $S M\cup S^+\partial M$.

Since $(M,\partial M,F)$ is compact, we can define a map
\[
\Psi:\{(t,y):\,y\in S^+\partial M,\, 0\leq t\leq {\mathfrak{t}}(y)\}\rightarrow SM, (t,y)\mapsto \varphi_t(y),
\]
where $\varphi_t$ is the geodesic flow of $F$.
Let $\widetilde{\mathfrak{t}}$ (resp. $\tilde{{i}}$) denote the $\mathfrak{t}$-function (resp. the cut value function) defined on $(M,\partial M,\widetilde{F})$, where $\widetilde{F}(y):=F(-y)$. Set
\[
{U}^-_M:=\{y\in S M: \tilde{\mathfrak{t}}(-y)< \tilde{{i}}(-y)\}.
\]
Since $y\in SM$ implies that $\widetilde{F}(-y)=1$, ${U}^-_M$ is well-defined. In particular, we have the following
\begin{lemma}\label{inject}
$\Psi|_{\mathfrak{N}_+}:{\mathfrak{N}_+}\rightarrow {U}^-_M\backslash U_\mathcal {Z}$ is a one-one map. Here, $\mathfrak{N}_+:=\{(t,y): y\in S^+\partial M, \, t\in (0,l(y))\}$, $U_\mathcal {Z}:=\{\varphi_t(y): y\in \mathcal {Z},\, t\in (0,l(y))\}$, and $l(y):=\min \{i(y),\mathfrak{t}(y)\}$.
\end{lemma}
\begin{proof}
Since $\overline{M}$ is compact, for each $y\in {U}^-_M$, $0<\tilde{\mathfrak{t}}(-y)< \tilde{{i}}(-y)<\infty$. Clearly, $\widetilde{\gamma}_{-y}(t)$, $0\leq t\leq \tilde{\mathfrak{t}}(-y)$ is a unit speed minimal geodesic in $(\overline{M},\widetilde{F})$. Set $Y:=-\dot{\widetilde{\gamma}}_{-y}(\tilde{\mathfrak{t}}(-y))$. Thus,
\[
F(Y)=\widetilde{F}(-Y)=\widetilde{F}(\dot{\widetilde{\gamma}}_{-y}(\tilde{\mathfrak{t}}(-y)))=1.
\]
It follows from Lemma \ref{in-in} that $g_{\mathbf{n}_+}(\mathbf{n}_+, Y)\geq 0$. Hence, $Y\in S^+\partial M\cup \mathcal {Z}$.

Let $d$ (resp. $\tilde{d}$) denote the distance function induced by $F$ (resp. $\widetilde{F}$). Let $p:=\pi(y)$ and $q:=\pi(Y)$.
Then $L_F(\gamma_Y([0,\tilde{\mathfrak{t}}(-y)]))=\tilde{\mathfrak{t}}(-y)=\tilde{d}(p,q)=d(q,p)$,
which implies that $i(Y)\geq \tilde{\mathfrak{t}}(-y)$. We claim that $i(Y)> \tilde{\mathfrak{t}}(-y)$. If not, then $p$ is the cut point of $q$ along $\gamma_Y$. If $p$ is also a conjugate point of $q$, then there exists a non-vanishing Jacobi field $J(t)$ along $\gamma_Y(t)$ such that $J(0)=0$ and $J(\tilde{\mathfrak{t}}(-y))=0$. It is easy to check that $\tilde{J}(t):=J(\tilde{\mathfrak{t}}(-y)-t)$ is a Jacobi field along $\widetilde{\gamma}_{-y}$ in $(\overline{M},\widetilde{F})$. Hence, $q$ is a conjugate point of $p$ along $\widetilde{\gamma}_{-y}$ in $(\overline{M},\widetilde{F})$, which contradicts $\tilde{\mathfrak{t}}(-y)<\tilde{i}(-y)$. Since $p$ is not a conjugate point of $q$, by the proof of \cite[Proposition 8.2.1]{BCS}, one can show that there exists another minimal geodesic from $q$ to $p$ in $(\overline{M},F)$. Thus, there exist two distinct minimal geodesic from $p$ to $q$ with the length $\tilde{\mathfrak{t}}(-y)$ in $(\overline{M},\widetilde{F})$, which also contradicts $\tilde{\mathfrak{t}}(-y)<\tilde{i}(-y)$. Hence, the claim is true, which implies that $\tilde{\mathfrak{t}}(-y)< \min\{\mathfrak{t}(Y),i(Y)\}=l(Y)$.

From above, we show that for each $y\in {U}^-_M$, there exist $Y\in S^+\partial M\cup \mathcal {Z}$ and $t:=\tilde{\mathfrak{t}}(-y)<l(Y)$ such that $y=\Psi(t,Y)$. Let $N_\mathcal {Z}:=\{(t,y): y\in \mathcal {Z}, t\in (0,l(y))\}$. Then $\Psi|_{\mathfrak{N}_+\cup N_\mathcal {Z}}:\mathfrak{N}_+\cup N_\mathcal {Z}\rightarrow {U}^-_M$ is subjective. Since $\Psi$ is injective, we are done by $\Psi(N_\mathcal {Z})=U_\mathcal {Z}$.
\end{proof}

Given any measure $d\mu$ on $M$, the induced volume forms on $\partial M$ by $\mathbf{n}_\pm$ are defined by $d\A_\pm:=i^*(\mathbf{n}_\pm\rfloor d\mu)$, where $i:\partial M\hookrightarrow M$ is the inclusion map (cf. \cite{Sh1}).
Now we have the following Santal\'o type formulas.
\begin{theorem}\label{fII}Let $(M,\partial M,F,d\mu)$ be a compact Finsler manifold with smooth boundary.
Thus,
for all integral function $f$ on $SM$, we have
\begin{align*}
&\int_{\mathcal {V}^-_M}f dV_{SM}=\int_{y\in S^+\partial M} e^{\tau(y)}g_{\mathbf{n}_+}(\mathbf{n}_+,y)d\chi_+(y)\int^{l(y)}_0f(\varphi_t(y))dt,\tag{1}\label{th1i}\\
&\int_{\mathcal {V}^+_M}f dV_{SM}=\int_{y\in S^-\partial M} e^{\tau(y)}g_{\mathbf{n}_-}(\mathbf{n}_-,y)d\chi_-(y)\int_0^{\tilde{l}(-y)}f(\varphi_{-t}(y))dt,\tag{2}\label{th1ii}
\end{align*}
where $\mathcal {V}^-_M:=\{y\in SM: \tilde{\mathfrak{t}}(-y)\leq  \tilde{{i}}(-y)\}$, $\mathcal {V}^+_M:=\{y\in SM: {\mathfrak{t}}(y)\leq  {{i}}(y)\}$ and $d\chi_\pm(y)=d\A_\pm(\pi(y))\wedge d\nu_{\pi(y)}(y)$.
\end{theorem}
\begin{proof}
(1). Given any $y\in S^+\partial M$. We identify $T_y(S^+\partial M)$ with its image in $T_{(0,y)}(\mathbb{R}\times S^+\partial M
)$. Since $\Psi_{*(0,y)}(X)=X$, $\forall X\in
T_y(S^+\partial M)$, we have
\[
\Psi^*(d\chi_+(y))\equiv d{\chi_+}|_{(0,y)}\ \ (\text{mod } dt).\tag{3.1}\label{3.1}
\]

We claim that $[\Psi^*\pi^*d\rho]|_{(0,y)}\equiv 0\ (\text{mod }dt)$. In fact, for each $X\in
T_y(S^+\partial M)$, there exists a curve
$\xi:[0,+\varepsilon)\rightarrow S^+\partial M$ with
$\xi(0)=y$ and $\dot{\xi}(0)=X$.
Thus,
\[
\langle X, \Psi^*\pi^*d\rho\rangle|_{(0,y)}=\langle {\pi}_*\left(\Psi_{*(0,y)}X\right),d\rho\rangle=\langle
{\pi}_*X,d\rho\rangle=\left.\frac{d}{ds}\right|_{s=0}\rho(\pi(\xi(s)))=0.
\]
The claim is true. Lemma \ref{in-in} now yields
\begin{align*}
\left[\Psi^*\pi^*_1d\rho\right]|_{(0,y)}&=\left\langle\frac{\partial}{\partial
t},\Psi^*\pi^*_1d\rho\right\rangle_{(0,y)}
dt\\
&=\left(\left.\frac{d}{dt}\right|_{t=0^+}\rho\circ
\gamma_y(t)\right)dt=g_{\mathbf{n}_+}(\mathbf{n}_+,y)dt.\tag{3.2}\label{3.2}
\end{align*}

Define a function $\eta\in C^\infty(\mathbb{R}\times S^+\partial M)$ by $\Psi^*(dV_{SM})=\eta\cdot \beta$, where
$\beta|_{(t,y)}=dt\wedge d\chi_+(y)$ is a $(2n-1)$ form on $\mathbb{R}\times S^+\partial M$. It is easy to check that $\eta(t,y)=\eta(0,y)$ (cf. \cite[Lemma 5.6]{ZY}). By the co-area formula (see \cite[Theorem 3.3.1]{Sh1}), (\ref{3.1}) and
(\ref{3.2}), we have
\begin{align*}
[\eta dt\wedge d\chi_+]|_{(0,y)}&=\Psi^*(dV_{SM}(y))=\Psi^*[e^{\tau(
y)}\pi^*(d\mu)(y)\wedge d\nu_{\pi(y)}(y)]\\
&=\Psi^*[e^{\tau(
y)}\pi^*(d\rho\wedge d\A_+)(y)\wedge d\nu_{\pi(y)}(y)]\\
&=[e^{\tau( y)}g_{\mathbf{n}_+}(\mathbf{n}_+,y)dt\wedge d\chi_+]|_{(0,y)},
\end{align*}
that is, $\eta(0,y)=e^{\tau( y)}g_{\mathbf{n}_+}(\mathbf{n}_+,y)$. It
follows from the definition of $\eta$ that
\[
\Psi^*(dV_{SM}(\varphi_t(y)))=e^{\tau(y)}g_{\mathbf{n}_+}(\mathbf{n}_+,y)dt\wedge
d\chi,\tag{3.3}\label{3.3}
\]
which implies that $\Psi$ is of maximal rank. Hence, Lemma \ref{inject} yields that
$\Psi|_{\mathfrak{N}_+}$ is a diffeomorphism.

Let $\mathscr{N}:=\{y\in
SM:\tilde{\mathfrak{t}}(-y)=\tilde{i}({-y})\}$. Thus, $\mathcal
{V}_M^-={U}_M^-\cup \mathscr{N}$. By an argument similar
to the proof of Lemma \ref{inject}, one has
$\mathscr{N}\subset\{\varphi_{l(y)}y: y\in S^+\partial M\cup
\mathcal {Z},\,l(y)=i(y) \}$, which implies that $\mathscr{N}$ has
measure zero with respect to $dV_{SM}$. Also note that $V_{SM}(U^-_M\backslash\Psi(\mathfrak{N}_+))=V_{SM}(U_\mathcal {Z})=0$.
Hence, by (\ref{3.3}), we have
\begin{align*}
\int_{\mathcal {V}^-_M}fdV_{SM}&=\int_{U^-_M}fdV_{SM}\\
&=\int_{\Psi
(\mathfrak{N}_+)}fdV_{SM}=\int_{\mathfrak{N}_+}\Psi^*(fdV_{SM})\\
&=\int_{S^+\partial M}e^{\tau(y)}g_{\mathbf{n}_+}(\mathbf{n}_+,y)d\chi(y)\int^{l(y)}_0f(\varphi_t(y))dt.
\end{align*}

(2). By considering $(M,\partial M,\widetilde{F})$ and using the formula (1),  we have
\[
\int_{y\in\widetilde{\mathcal {V}^-_M}}f(-y) d\widetilde{V_{SM}}(y)=\int_{y\in \widetilde{S^+\partial M}} e^{\tilde{\tau}(y)}\tilde{g}_{\tilde{\mathbf{n}}_+}(\tilde{\mathbf{n}}_+,y)d\widetilde{\chi}_+(y)\int^{\tilde{l}(y)}_0f(-\tilde{\varphi}_t(y))dt,
\]
where the quantities $\tilde{*}$ denote the quantities $*$  defined by $\widetilde{F}$. Note that $\tilde{\mathbf{n}}_+=-{\mathbf{n}}_-$ and $-\tilde{\varphi}_t(y)=\varphi_{-t}(-y)$, $0\leq t\leq \tilde{l}(y)$. The formula (2) now follows from the transformation $y\mapsto -y$.\end{proof}

\section{A universal lower bound for the first eigenvalue of the nonlinear Laplacian}

\begin{definition}[\cite{GS,Sh3}]
Let $(M,F,d\mu)$ be a compact Finsler manifold. Denote
$\mathscr{H}_0(M,d\mu)$ by
\[
\mathscr{H}_0(M,d\mu):= \left \{
\begin{array}{lll}
&\{f\in W^1_2(M):\,\int_M fd\mu=0\}, &\partial M=\emptyset,\\
\\
&\{f\in W^1_2(M):\,f|_{\partial M}=0\},&\partial M\neq\emptyset.
\end{array}
\right.
\]

Define the canonical energy functional $E_{d\mu}$ on
$\mathscr{H}_0(M,d\mu)-\{0\}$ by
\[
E_{d\mu}(u):=\frac{\int_M F^*(du)^2d\mu}{\int_M u^2 d\mu}.
\]

$\lambda$ is an eigenvalue if there is a function
$u\in\mathscr{H}_0(M,d\mu)-\{0\}$ such that $d_u E_{d\mu}=0$ with
$\lambda=E_{d\mu}(u)$. In this case, $u$ is called an eigenfunction
corresponding to $\lambda$. The first eigenvalue
$\lambda_1(M,d\mu)$ is defined by
\[\lambda_1(M,d\mu):=\underset{u\in
\mathscr{H}_0(M,d\mu)-\{0\} }{\inf}E_{d\mu}(u),
\]
which is the smallest positive critical value of $E_{d\mu}$.
\end{definition}
\begin{remark}
$u$ is an eigenfunction
corresponding to $\lambda$ if and only if
\[
\Delta_{d\mu} u +\lambda u=0 \text{ (in the weak sense)},
\]
where $\Delta_{d\mu}$ is the nonlinear Laplacian introduced by Shen \cite{GS,Sh1,Sh3}. It should be noted that $\Delta_{d\mu}$ is dependent on the choice of $d\mu$.
\end{remark}

\begin{proposition}\label{gradestimate}Let $(M,F)$ be a Finsler $n$-manifold. Then for any $p\in M$ and
$f\in C^\infty(M)$, we have
\[
F^*(df|_{p})^2\geq \frac{n}{c_{n-1}\Lambda^{n+1}_F(p)}\int_{S_pM}\langle y,df\rangle^2 d\nu_p(y),\tag{4.1}\label{5.**}
\]
with equality if and only if $F(p,\cdot)$ is a Eucildean norm.
\end{proposition}
\begin{proof}
Without loss of generality, we may suppose $df|_p\neq0$. Set $B_pM:=\{y\in T_pM: F(p,y)<1\}$. By \cite{W}, one can choose a $g_{\nabla f}$-orthnormal basis $\{e_i\}$ of $T_pM$ such that $e_n=\nabla f/ F(\nabla f)$ and $\deg g_{ij}(p,y)\leq \Lambda_F^n(p)$. Let $\{y^i\}$ denote the corresponding coordinates. By Stokes' formula, we have
\begin{align*}
&\int_{S_pM}\langle y,df\rangle^2 d\nu_p(y)\\
\leq &\Lambda_F^{\frac{n}{2}}(p) F^2(\nabla f)\int_{S_pM}\left({y^n}\right)^2\overset{n}{\sum_{k=1}}(-1)^{k-1}y^k dy^1\wedge \cdots \wedge \widehat{dy^k}\wedge \cdots \wedge dy^n\\
=&(n+2) \Lambda_F^{\frac{n}{2}}(p) F^2(\nabla f)\int_{B_pM} (y^n)^2dy^1\wedge \cdots \wedge dy^n\\
\leq &(n+2) \Lambda_F^{\frac{n}{2}}(p) F^2(\nabla f)\int_{\mathbb{B}^n(\sqrt{\Lambda_F(p)})} (y^n)^2dy^1\wedge \cdots \wedge dy^n\tag{4.2}\label{5.1}\\
=&\frac{c_{n-1}}{n}\Lambda_F^{n+1}(p)F^2(\nabla f)
\end{align*}
If equality holds in (\ref{5.**}), then it follows from (4.2) that $B_pM=\mathbb{B}^n(\sqrt{\Lambda_F(p)})$. Namely, $F(y)=1$ if and only if $g_{\nabla f}(y,y)=\Lambda_F(p)$. In particular,
$1=F(e_n)=g_{\nabla f}(e_n,e_n)=\Lambda_F(p)$, which implies that $F(p,\cdot)$ is a Eucildean norm.\end{proof}

\begin{theorem}\label{firsteig}
Let $(M,\partial M,F)$ be a compact Finsler $n$-manifold with smooth boundary such that every geodesic ray in $(M,{F})$ minimizes distance up to the point that it intersects $\partial M$.  Then
\[
\lambda_1(M,d\mu)\geq\left \{
\begin{array}{lll}
&\frac{\lambda_1(\mathbb{S}_D^+)}{\Lambda_F^{4n+1}},& d\mu=d\mu_{BH},\\
\\\tag{4.3}\label{inequ}
&\frac{\lambda_1(\mathbb{S}_D^+)}{\Lambda_F^{2n+1}},& d\mu=d\mu_{HT},
\end{array}
\right.
\]
where $D:=\diam(M)$ and $\mathbb{S}_D^+$ denotes the $n$-dimensional Riemannian hemisphere of the constant sectional curvature sphere having diameter equal to $D$. The equality holds if and only if $(M,F)$ is isometric to $\mathbb{S}_D^+$.
\end{theorem}
\begin{proof}Lemma 2.1 yields that
\[
\int_{S_pM}e^{\tau(y)}d\nu_p(y)=c_{n-1}\frac{\sigma_{HT}(p)}{\sigma(p)}\geq\left \{
\begin{array}{lll}
&\frac{c_{n-1}}{\Lambda_F^{2n}},& d\mu=d\mu_{BH},\\
\\ \tag{4.4}\label{4.3}
&c_{n-1},& d\mu=d\mu_{HT}.
\end{array}
\right.
\]

Since $\mathcal {V}^+_M=SM$, Theorem \ref{fII} together with Proposition \ref{gradestimate} and (\ref{4.3}) then yields
\begin{align*}
&\int_M F^{*2}(df)d\mu\\
\geq &\frac{n}{c_{n-1}\Lambda_F^{n+1}}\int_M d\mu(p)\int_{S_pM}\langle y,df\rangle^2 d\nu_p(y)\tag{4.5}\label{5.2}\\
=&\frac{n}{c_{n-1}\Lambda_F^{n+1}}\int_{SM}e^{-\tau(y)}\langle y,df\rangle^2 dV_{SM}(y)\\
=&\frac{n}{c_{n-1}\Lambda_F^{n+1}}\int_{y\in S^-\partial M}e^{\tau(y)}g_{\mathbf{n}_-}(\mathbf{n}_-,y)d\chi_-(y)\int_{-\tilde{l}(-y)}^0 e^{-\tau(\varphi_t(y))}\langle \varphi_t(y),df\rangle^2 dt\\
\geq & \frac{n}{c_{n-1}\Lambda_F^{2n+1}}\int_{y\in S^-\partial M}e^{\tau(y)}g_{\mathbf{n}_-}(\mathbf{n}_-,y)d\chi_-(y)\int_{-\tilde{l}(-y)}^0\left(\frac{d}{dt}f(\gamma_y(t))\right)^2 dt\\
\geq &\frac{n}{c_{n-1}\Lambda_F^{2n+1}}\int_{y\in S^-\partial M}e^{\tau(y)}g_{\mathbf{n}_-}(\mathbf{n}_-,y)d\chi_-(y)\int_{-\tilde{l}(-y)}^0\left(\frac{\pi}{\tilde{l}(-y)}\right)^2f^2(\gamma_y(t))dt\\
\geq &\frac{n}{c_{n-1}\Lambda_F^{2n+1}}\left(\frac{\pi}{D}\right)^2\int_{SM}f^2(\pi(y))dV_{SM}(y)\\
\geq &\left \{
\begin{array}{lll}
&\frac{\lambda_1(\mathbb{S}^+_D)}{\Lambda_F^{4n+1}}\int_Mf^2d\mu,& d\mu=d\mu_{BH},\\
\\
&\frac{\lambda_1(\mathbb{S}^+_D)}{\Lambda_F^{2n+1}}\int_Mf^2d\mu,& d\mu=d\mu_{HT}.
\end{array}
\right.
\end{align*}
If we have equality in (\ref{inequ}), then (\ref{5.2}) together with Proposition \ref{gradestimate} implies $\Lambda_F=1$. Hence, $(M,F)$ is a Riemannian manifold and $\lambda_1(M)=\lambda_1(\mathbb{S}^+_D)$. By the standard argument (see \cite[p.131]{C2} or \cite{Cr}), one can show that $(M,F)$ is isometric to $\mathbb{S}^+_D$.
\end{proof}

In \cite{Sh3}, Shen shows that the first eigenvalue of a forward metric ball is bounded from above by a constant depending only on the dimension and lower bounds on the Ricci curvature and the S-curvature. From Theorem \ref{firsteig}, we obtain a lower bound for the first eigenvalue of a forward metric ball.

\begin{corollary}Let $(M,F,d\mu)$ be a forward complete Finsler $n$-manifold of injectivity radius $\mathfrak{i}_M$.
For any $0<r<\mathfrak{i}_M/(1+\sqrt{\Lambda_F})$ and any $p\in M$, we have
\[
\lambda_1(B^+_p(r))\geq\left \{
\begin{array}{lll}
&\frac{\lambda_1\left(\mathbb{S}^+_{2\sqrt{\Lambda_F}r}\right)}{\Lambda_F^{4n+1}},& d\mu=d\mu_{BH},\\
\\
&\frac{\lambda_1\left(\mathbb{S}^+_{2\sqrt{\Lambda_F}r}\right)}{\Lambda_F^{2n+1}},& d\mu=d\mu_{HT}.
\end{array}
\right.
\]
with equality if and only if $B^+_p(r)$ is isometric to $\mathbb{S}^+_{2r}$.
\end{corollary}

\section{Croke type isoperimetric inequalities }
In this section, we shall establish Theorem \ref{t3} and give some applications.

\begin{lemma}\label{es1}
For each $x\in \partial M$, we have
\[
\int_{S^\sharp_x\partial M}g_{\mathbf{n}_\sharp}(\mathbf{n}_\sharp,y)e^{\tau(y)}d\nu_x(y)\leq \frac{c_{n-2}}{n-1}\Lambda_F^{2n+\frac12}(x),
\]
with equality if and only if $F(x,\cdot)$ is a Euclidean norm.
Here, "$\sharp$" denotes either "$+$" or "$-$", and $S^\sharp_x\partial M:=\{y\in S_xM:\,g_{\mathbf{n}_\sharp}(\mathbf{n}_\sharp,y)>0\}$.
\end{lemma}
\begin{proof} Suppose $\sharp=+$. By \cite{W}, one can choose a $g_{\mathbf{n}_+}$-orthnormal basis $\{e_i\}$ of $T_xM$ such that $e_n=\mathbf{n}_+$ and $\det g_{ij}(x,y)\leq \Lambda_F^n(x)$. Let $\{y^i\}$ be the corresponding coordinates. Set $\|\cdot\|:=\sqrt{g_{\mathbf{n}_+}(\cdot, \cdot)}$. Define
\begin{align*}
&B^+_x:=\{y\in T_xM: F(y)<1,\, y^n>0\}, \ B^+_{x,r}:=\{y\in T_xM: F(y)=1,\, y^n=r\}\\
&\mathbb{B}_{x,r}(s):=\{y\in T_xM: y^n=r,\, \|y^\alpha e_\alpha\|<s\}, \ \varpi:=g_{\mathbf{n}_+}(\mathbf{n}_+,y)e^{\tau(y)}d\nu_p(y).
\end{align*}
For each $y\in B^+_x$, $y^n=g_{\mathbf{n}_+}(\mathbf{n}_+,y)\leq F(\mathbf{n}_+)F(y)\leq 1$.
Stokes' formula together with Lemma 2.1 then yields
\begin{align*}
\int_{S_x^+\partial M}\varpi&\leq\Lambda_F^{3n/2}(x)\int_{S_x^+\partial M}y^n\overset{n}{\sum_{k=1}}(-1)^{k-1}y^k dy^1\wedge \cdots \wedge \widehat{dy^k}\wedge \cdots \wedge dy^n\\
&= (n+1)\Lambda_F^{\frac{3n}{2}}(x)\int_{B^+_x}y^n dy^1\wedge\cdots\wedge dy^n\\
&=(n+1)\Lambda_F^{\frac{3n}{2}}(x)\int^1_0\vo(B^+_{x,y^n})y^ndy^n\\
&\leq (n+1)\Lambda_F^{\frac{3n}{2}}(x)\int^{\sqrt{\Lambda_F(x)}}_0 \vo \left(\mathbb{B}_{x,y^n}(\sqrt{\Lambda_F(x)-(y^n)^2})\right)y^ndy^n\\
&=\frac{c_{n-2}}{n-1}\Lambda_F^{2n+\frac12}(x),
\end{align*}
with equality if and only if $\Lambda_F(x)=1$, i.e., $F(x,\cdot)$ is a Euclidean norm.

Suppose $\sharp=-$. Note that $\Lambda_F(x)=\Lambda_{\widetilde{F}}(x)$.  Using the same method as in Theorem \ref{fII}, one can get the formula.
\end{proof}

Given any point $x\in M$, let $(r,y)$ denote the polar coordinates about $x$. Set $\mathscr{F}(r,y)=e^{\tau(\gamma_y(r))}\hat{\sigma}_x(r,y)$, where $d\mu|_{(r,y)}=:\hat{\sigma}_x(r,y)dr\wedge d\nu_x(y)$. Then we have the following inequality of Berger-Kazdan type \cite[Theorem 1.3]{ZY}
\begin{lemma}[\cite{ZY}]\label{B-K}
Let $(M,F)$ be a compact Finsler $n$-manifold. For each $y\in SM$
and $0<t\leq l\leq i_y$, we have
\[
\int^{l}_0 dr\int^{l-r}_0\mathscr{F}(t, \varphi_r(y))\,dt\geq
\frac{\pi c_{n}}{2c_{n-1}}\left(\frac{l}{\pi}\right)^{n+1},
\]
with equality if and only if
\[
R_{\dot{\gamma}_y(t)}(\cdot,\dot{\gamma}_y(t))\dot{\gamma}_y(t)=\left(\frac{\pi}{l}\right)^{2}\id,\
0\leq t\leq l,
\]
where $R$ is the (Riemannian) curvature tensor acting on
$\dot{\gamma}_y(t)^\bot$.
\end{lemma}

Now we have the following theorem.
\begin{theorem}\label{t32}
Let $(M,\partial M,F,d\mu)$ be a compact Finsler $n$-manifold with smooth boundary, where $d\mu$ is either the Busemann-Hausdorff measure or the Holmes-Thompson measure. Set
\[
\omega:=\inf_{x\in M}\min\{\omega^+_x,\,\omega^-_x\}=\min \left\{\inf_{x\in M}\omega^+_x,\inf_{x\in M}\omega^-_x\,\right\},
\]
where $\omega_x^\pm:=c_{n-1}^{-1}\int_{U_x^\pm}e^{\tau(y)}d\nu_x(y)$ and $U_x^\pm:=\pi|_{\mathcal {V}^\pm_M}^{-1}(x)$.
Then

(1)
\[
\frac{\A_\pm(\partial M)}{\mu(M)}\geq \frac{(n-1)c_{n-1}\,\omega}{c_{n-2}\,D\,\Lambda_F^{2n+\frac12}},
\]
where $D:=\diam(M)$.

(2)
\[
\frac{\A_\pm(\partial M)}{\mu(M)^{1-\frac1n}}\geq \frac{c_{n-1}\omega^{1+\frac1n}}{(c_n/2)^{1-\frac1n}\Lambda_F^{2n+\frac52}},\tag{5.1}\label{Crokii}
\]
with equality if and only if $(M, F)$ is a Riemannian hemisphere of a constant sectional curvature sphere.
\end{theorem}

\begin{proof}(1) Theorem \ref{t1} together with Lemma \ref{es1} furnishes
\begin{align*}
c_{n-1}\omega \mu(M)&\leq c_{n-1}\int_M \omega^\mp_x d\mu(x)=\int_{x\in M}d\mu(x)\int_{U^\mp_x}e^{\tau(y)}d\nu_x(y)\\&=V_{SM}(\mathcal {V}^\mp_M)
\leq D\int_{S^\pm\partial M}e^{\tau(y)}g_{\mathbf{n}_\pm}(\mathbf{n}_\pm,y)d\chi_\pm(y)\\
&\leq D\A_\pm(\partial M)\frac{c_{n-2}}{n-1}\Lambda_F^{2n+\frac12}.
\end{align*}

(2) For each $y\in S^+\partial M$, $l(\varphi_t(y))\geq l(y)-t$, for any $0\leq t\leq l(y)$. By Theorem \ref{t1}, Lemma 2.1, Theorem \ref{B-K} and H\"older's inequality, we have
\begin{align*}
&\mu^2(M)\\
=&\int_M d\mu(x)\int_{S_xM}d\nu_x(y)\int^{l(y)}_0 \hat{\sigma}_x(r,y)dr=\int_{SM}dV_{SM}(y)\int^{l(y)}_0e^{-\tau(y)}\hat{\sigma}_x(r,y)dr\\
\geq& \int_{\mathcal {V}^-_M}dV_{SM}(y)\int^{l(y)}_0e^{-\tau(y)}\hat{\sigma}_x(r,y)dr\\
=&\int_{S^+\partial M}e^{\tau(y)}g_{\mathbf{n}_+}(\mathbf{n}_+,y)d\chi_+(y)\int^{l(y)}_0dt\int^{l(\varphi_t(y))}_0e^{-\tau(\varphi_t(y))}\hat{\sigma}_x(r,\varphi_t(y))dr\\
\geq& \Lambda_F^{-2n} \int_{S^+\partial M}e^{\tau(y)}g_{\mathbf{n}_+}(\mathbf{n}_+,y)d\chi_+(y)\int^{l(y)}_0dt\int^{l(y)-t}_0\mathscr{F}(r,\varphi_t(y))dr\\
\geq&\frac{c_n}{2c_{n-1}\pi^n\Lambda_F^{2n}}\int_{S^+\partial M}l(y)^{n+1}e^{\tau(y)}g_{\mathbf{n}_+}(\mathbf{n}_+,y)d\chi_+(y)\\
\geq&  \frac{c_n}{2c_{n-1}\pi^n\Lambda_F^{2n}}\left(\int_{S^+\partial M} l(y)e^{\tau(y)}g_{\mathbf{n}_+}(\mathbf{n}_+,y)d\chi_+(y)\right)^{n+1}\left(\int_{S^+\partial M}e^{\tau(y)}g_{\mathbf{n}_+}(\mathbf{n}_+,y)d\chi_+(y) \right)^{-n}
\end{align*}
\begin{align*}
\geq & \frac{c_n}{2c_{n-1}\pi^n\Lambda_F^{2n}}V_{SM}(\mathcal {V}^-_M)^{n+1}\left(\frac{n-1}{c_{n-2}\A_+(\partial M)\Lambda^{2n+\frac12}_F}\right)^n\tag{5.2}\label{4.5}\\
\geq &\frac{(c_{n-1})^n\omega^{n+1}\mu(M)^{n+1}}{(c_n/2)^{n-1}\A_+^n(\partial M)\Lambda_F^{(2n+\frac52)n}}.
\end{align*}
That is,
\[
\frac{\A_+(\partial M)}{\mu(M)^{1-\frac1n}}\geq \frac{c_{n-1}\omega^{1+\frac1n}}{(c_n/2)^{1-\frac1n}\Lambda_F^{2n+\frac52}}.\tag{5.3}\label{A+hol}
\]
Let $\widetilde{\A}_\pm$ and $\tilde{\omega}$ be define as before on $(M,\widetilde{F})$. It is easy to check that  $\widetilde{\A}_\pm(\partial M)={\A}_\mp(\partial M)$ and $\tilde{\omega}=\omega$. From above, we obtain
\[
\frac{\A_-(\partial M)}{\mu(M)^{1-\frac1n}}=\frac{\widetilde{\A}_+(\partial M)}{\tilde{\mu}(M)^{1-\frac1n}}\geq \frac{c_{n-1}{\omega}^{1+\frac1n}}{(c_n/2)^{1-\frac1n}\Lambda_{{F}}^{2n+\frac52}}.\tag{5.4}\label{5.4}
\]
(\ref{A+hol}) together with (\ref{5.4}) yields (\ref{Crokii}).

Suppose that equality holds in (\ref{Crokii}). Then we have equality in (\ref{A+hol}) or (\ref{5.4}).
It follows from (\ref{4.5}) and Lemma \ref{es1} that $1=\Lambda_F=\Lambda_{\widetilde{F}}$. Hence, $F$ is an Riemannian metric and (\ref{Crokii}) becomes
\[
\frac{\A(\partial M)}{\mu(M)^{1-\frac1n}}= \frac{c_{n-1}\omega^{1+\frac1n}}{(c_n/2)^{1-\frac1n}}.
\]
Since $\mathcal {V}^-_M=SM$, $\mathfrak{t}(y)\leq i_y$, for all $y\in S\overline{M}$. H\"older's inequality implies $l(y)$ is constant, say, equal to $l$, on
all of $S^+\partial M$. Hence, $\mathfrak{t}(y)=l$, for all $y\in S^+\partial M$. And Theorem \ref{B-K} yields $M$ has constant sectional curvature
equal to $(\pi/l)^2$, i.e., $M$ is a hemisphere.
\end{proof}
From above, it is easy to see that Theorem \ref{t32} becomes Croke's isoperimetric inequality \cite{Cr} if $\Lambda_F=1$.
In the Finslerian case, the upper bound on $\Lambda_F$ in Theorem \ref{t32} is very important as the following example shows.

\begin{example}\label{Funk}
 Let $\mathbb{B}^n$ be the unit open ball in $\mathbb{R}^n$ equipped with a Funk metric $F$, that is,
\[
F(x,y)=\frac{\sqrt{(1-|x|^2)|y|^2+(x\cdot y)^2}+x\cdot y}{1-|x|^2},
\]
where $"|\cdot|"$ (resp. $" \cdot "$) denotes the Euclidean norm (resp. inner product). For $r\in (0,1)$, set $\Omega_r:=\{x\in \mathbb{B}^n: \, |x|<r\}$. Then $(\Omega_r,\partial\Omega_r,F|_{\overline{\Omega}_r})$ is a compact Finsler manifold with smooth boundary. By directly computing, we have $\mu_{BH}(\Omega_r)=\frac{c_{n-1}}{n}r^n$ and $\A_\pm(\partial\Omega_r)=c_{n-1}(1\pm r)r^{n-1}$, where $d\A_\pm$ are induced by $d\mu_{BH}$.
Clearly,
\[
\lim_{r\rightarrow 1}\frac{\A_+(\partial\Omega_r)}{\A_-(\partial\Omega_r)}=+\infty.
\]
Note that
\[
\Lambda_{F|_{\overline{\Omega}_r}}=\left(\frac{1+r}{1-r}\right)^2,\ \diam(\Omega_r)=\log\left( \frac{1+r}{1-r}\right).
\]
For any $x\in \Omega_r$,
\[
 \omega_x^\pm=\frac{1}{(1-|x|^2)^{\frac{n+1}2}}\geq 1,\ \text{i.e., } \omega=1.
\]
Hence,  we have
\begin{align*}
\frac{\A_\pm(\partial\Omega_r)}{\mu_{BH}(\Omega_r)}> \frac{(n-1)c_{n-1}\,\omega}{c_{n-2}\diam(\Omega_r)\,\Lambda_{F|_{\overline{\Omega}_r}}^{2n+\frac12}},\frac{\A_\pm(\partial\Omega_r)}{\mu_{BH}(\Omega_r)^{1-\frac1n}}> \frac{c_{n-1}\omega^{1+\frac1n}}{(c_n/2)^{1-\frac1n}\Lambda_{F|_{\overline{\Omega}_r}}^{2n+\frac52}}.
\end{align*}
In particular,
\[
\lim_{r\rightarrow 1}\Lambda_{F|_{\overline{\Omega}_r}}=+\infty,\ \lim_{r\rightarrow 1}\frac{\A_-(\partial\Omega_r)}{\mu_{BH}(\Omega_r)}=\lim_{r\rightarrow 1}\frac{\A_-(\partial\Omega_r)}{\mu_{BH}(\Omega_r)^{1-\frac1n}}=0.
\]
\end{example}

Before giving some applications of Theorem \ref{t32}, we introduce the definitions of the Sobolev constant, Cheeger's constant and the isoperimetric constant of a closed Finsler manifold.
\begin{definition}\label{def2}
Let $(M,F,d\mu)$ be a closed Finsler manifold.
The Sobolev constant $\mathcal {S}(M,d\mu)$ is defined as
\[
\mathcal {S}(M,d\mu):=\inf_{f\in C^\infty(M)}\frac{\left\{\int_MF^*(df)d\mu\right\}^n}{\inf_{\alpha\in \mathbb{R}}\left\{\int_M |f-\alpha|^{\frac{n}{n-1}}d\mu\right\}^{n-1}}.
\]
Cheeger's constant
${\mathbbm{h}}(M,d\mu)$ and the isoperimetric constant $\mathbb{I}(M,d\mu)$ are defined by
\begin{align*}
\mathbbm{h}(M,d\mu):=\underset{\Gamma}{\inf}\frac{\min\{\A_\pm(\Gamma)\}}{\min\{\mu(M_1),\mu(M_2)\}},\
\mathbb{I}(M,d\mu):=\inf_\Gamma\frac{\min\{\A_\pm(\Gamma)\}^n}{\{\min\{\mu(M_1),\,\mu(M_2)\}\}^{n-1}},
\end{align*}
where $\Gamma$ varies over compact $(n-1)$-dimensional submanifolds
of $M$ which divide $M$ into disjoint open submanifolds $M_1$, $M_2$
of $M$ with common boundary $\partial M_1=\partial M_2=\Gamma$.
\end{definition}

\begin{remark}\label{constant}
By using the co-area formula (cf. \cite[Theorem 3.3.1]{Sh1}) and the same argument as in \cite{L}, one can obtain a Cheeger type inequality
\[
\lambda_1(M,d\mu)\geq\frac{\mathbbm{h}^2(M,d\mu)}{4\lambda_F^2} .
\]
And we also have a Federer-Fleming type inequality (see Proposition 6.1 below), i.e.,
\[
\mathbb{I}(M,d\mu)\leq \mathcal {S}(M,d\mu)\leq 2 \mathbb{I}(M,d\mu).
\]
\end{remark}

\begin{corollary}Let $(M,F,d\mu)$ be a closed Finlser $n$-manifold with $\mathbf{Ric}\geq (n-1)k$, where $d\mu$ denotes either the Busemann-Hausdorff measure or the Holmes-Thompson measure. Then
\begin{align*}
\lambda_1(M,d\mu)&\geq \left[\frac{(n-1)\mu(M)}{4c_{n-2}\Lambda_F^{4n+1}\diam (M)\int^{\diam (M)}_0\mathfrak{s}_k^{n-1}(t)dt}\right]^2,\\
\mathcal {S}(M,d\mu)&\geq \frac{\mu(M)^{n+1}}{4 c_{n-1}(c_n)^{n-1}\Lambda_F^{4n^2+\frac{9n}{2}}\left(\int^{\diam(M)}_0\mathfrak{s}_k^{n-1}(t)dt\right)^{n+1}}.
\end{align*}
Hence, both $\lambda_1(M,d\mu)$ and $\mathcal {S}(M,d\mu)$ can be bounded from below in terms of the diameter, volume, uniform constant and a lower bound for the Ricci curvature.
\end{corollary}
\begin{proof}
\noindent \textbf{Step 1.} Let $\Gamma$ be any $(n-1)$-dimensional compact submanifold of $M$ dividing $M$ into two open submanifolds $M_1$ and $M_2$, such that $\partial M_1=\partial M_2=\Gamma$. Given $x\in M_1$, let
\[
O_x:=\{q\in M: \exists \, y\in U^-_x \text{ such that }q=\widetilde{\gamma}_{-y}(t), t\in(0, \tilde{i}(-y)]\},
 \]
where $\widetilde{\gamma}_{-y}(t)$ is the geodesic in $(M,\widetilde{F})$ with $\dot{\widetilde{\gamma}}_{-y}(0)=-y$.

For any $q\in M_2$, there exists a minimal unit speed geodesic, say $\widetilde{\gamma}_X(t)$, from $x$ to $q$. Clearly, $\widetilde{\gamma}_X(t)$ must hit the boundary and therefore, $\tilde{\mathfrak{t}}(X)\leq \tilde{i}(X)$. Since $F(-X)=\widetilde{F}(X)=1$, $q\in O_x$ which implies that $M_2\subset O_x$.

Note that $\widetilde{\mathbf{Ric}}\geq (n-1)k$, $\Lambda_{\widetilde{F}}= \Lambda_F$ and $d\tilde{\mu}=d\mu$. Hence, by Lemma 2.1 and the volume comparison theorem (cf. \cite[Theorem 3.1]{ZY}), we have
\begin{align*}
\mu(M_2)&=\tilde{\mu}(M_2)\leq \tilde{\mu}(O_x)=\int_{y\in U^-_x}d\tilde{\nu}_x(-y)\int^{\tilde{i}(-y)}_0\widetilde{\hat{\sigma}}_x(r,-y)dr\\
&\leq\Lambda_F^n\int_{y\in U^-_x}d\tilde{\nu}_x(-y)\int^{\tilde{i}(-y)}_0\mathfrak{s}^{n-1}_k(r)dr\\
&\leq c_{n-1} \Lambda_F^{2n} \omega^-_1(x) \int^{\diam(M)}_0\mathfrak{s}^{n-1}_k(r)dr.
\end{align*}
That is,
\[
\omega^-_i\geq \frac{\mu(M_j)}{c_{n-1}\Lambda_F^{2n}\int^{\diam(M)}_0\mathfrak{s}^{n-1}_k(r)dr}, \, i\neq j.
\]

Set $O'_x:=\{q\in M: \exists \, y\in U^+_x \text{ such that }q={\gamma}_{y}(t), t\in(0, {i}(y)]\}$. It is easy to see that $M_2\subset O'_x$. By the similar argument, one can show that
\[
\omega^+_i\geq \frac{\mu(M_j)}{c_{n-1}\Lambda_F^{2n}\int^{\diam(M)}_0\mathfrak{s}^{n-1}_k(t)dt}, \, i\neq j.
\]

\noindent \textbf{Step 2.} The inequalities above together with Theorem \ref{t32} yield
\begin{align*}
\mathbbm{h}(M,d\mu)&\geq \frac{(n-1)\mu(M)}{2c_{n-2}\Lambda_F^{4n+\frac12}\diam (M)\int^{\diam (M)}_0\mathfrak{s}_k^{n-1}(t)dt},\\
\mathbb{I}(M,d\mu)&\geq \frac{\mu(M)^{n+1}}{4 c_{n-1}(c_n)^{n-1}\Lambda_F^{4n^2+\frac{9n}{2}}\left(\int^{\diam(M)}_0\mathfrak{s}_k^{n-1}(t)dt\right)^{n+1}}.
\end{align*}
Corollary now follows from Remark \ref{constant}.
\end{proof}

\begin{corollary}\label{Balles}Let $(M,F,d\mu)$ be a closed Finsler $n$-manifold, where $d\mu$ is either the Busemann-Hausdorff measure or the Holmes-Thompson measure. Then for any $x\in M$ and $0<r< \mathfrak{i}_M/(1+\sqrt{\Lambda_F})$, we have
\begin{align*}
\mu(B_x^+(r))\geq \frac{C(n,\Lambda_F)}{n^n}r^n, \, \A_\pm(S^+_x(r))\geq \frac{C(n,\Lambda_F)}{n^{n-1}}r^{n-1}.
\end{align*}
\end{corollary}
\begin{proof}The similar argument as in Lemma \ref{inject} shows
$\mathfrak{i}_M=\tilde{\mathfrak{i}}_M$, where $\tilde{\mathfrak{i}}_M$ is the injectivity radius of $(M,\widetilde{F})$. Hence, $U_x^\pm=S_xM$ for all $x\in B_x^+(r)$.
By Theorem \ref{t32} and (\ref{4.3}), we have
\[
\frac{\frac{d}{dr}\mu(B^+_x(r))}{\mu(B^+_x(r))^{1-\frac1n}}=\frac{\A_-(S^+_x(r))}{\mu(B^+_x(r))^{1-\frac1n}}\geq C(n,\Lambda_F),
\]
which implies that
\[
\mu(B_x^+(r))\geq \frac{C(n,\Lambda_F)}{n^n}r^n.\tag{5.5}\label{4.8}
\]
Theorem \ref{t32} together with (\ref{4.8}) yields
\[
\A_\pm(S^+_x(r))\geq \frac{C(n,\Lambda_F)}{n^{n-1}}r^{n-1}.
\]
\end{proof}

In order to establish Theorem 1.5, let us recall some definitions and properties of general LGC spaces first. Refer to \cite{SZ1,Z} for more details.
\begin{definition}[\cite{SZ1,Z}]
A general metric space is a pair $(X,d)$, where $X$ is a set
and $d: X\times X \rightarrow \mathbb{R}^{+}\cup\{\infty\}$, called
a metric, is a function, satisfying the following two conditions: (a) $d(x,y)\geq 0, \mbox{with equality }\Leftrightarrow x=y;$
(b) $d(x,y)+d(y,z)\geq d(x,z).$
The reversibility $\lambda_X$ of a general metric space $(X,d)$ is defined by
$\lambda_X:=\sup_{x\neq y}\frac{d(x,y)}{d(y,x)}.$

A contractibility function $\rho:[0,r)\rightarrow [0,+\infty)$
is a function satisfying: (a) $\rho(0)=0$, (b) $\rho(\epsilon)\geq
\epsilon$, (c) $ \rho(\epsilon)\rightarrow 0$, as
$\epsilon\rightarrow 0$, (d) $ \rho$  is nondecreasing.
A general metric space $X$ is $\LGC(\rho)$ for some contractibility
function $\rho$, if for every $\epsilon\in [0,r)$ and $x\in X$, the
forward ball $B^+_x(\epsilon)$ is contractible inside
$B^+_x(\rho(\epsilon))$.
\end{definition}

\begin{lemma}[\cite{Z}]\label{homotopfin} Fix a function $N:(0,\alpha)\rightarrow
(0,\infty)$ with
\[\underset{\epsilon\rightarrow
0^+}{\lim\sup}\epsilon^nN(\epsilon)<\infty
\] and a contractibility
function $\rho:[0,r)\rightarrow [0,\infty)$. The class
\[
\mathscr{C}(N,\rho):=\{X\in \mathcal {M}^\delta: X \text{
is } \LGC(\rho), \ \Cov(X,\epsilon)\leq N(\epsilon)\text { for all
}\epsilon\in (0,\alpha)\}
\]
contains only finitely many homotopy types. Here, $\mathcal {M}^\delta$ denotes the collection of compact general metric spaces
with reversibility $\leq\delta$ and $\Cov(X; \epsilon)$ denotes the minimum number
of forward $\epsilon$-balls it takes to cover $X$.
\end{lemma}
Corollary \ref{Balles} together with Lemma \ref{homotopfin} yields the following
\begin{theorem}\label{Ya}
For any $n$ and positive numbers $i, V, \delta$, the class of closed
Finsler $n$-manifolds $(M,F)$ with injectivity radius
$\mathfrak{i}_M\geq i$, $\Lambda_F\leq \delta$ and $\mu(M) \leq V$,
contains at most finitely many homotopy types. Here, $\mu(M)$ is either the Busemann-Hausdorff volume or the Holmes-Thompson volume of $M$.
\end{theorem}
\begin{proof}Let $c_M$ denote the contractibility radius of $(M,F)$ (cf. \cite{Z}).
Since $c_M\geq \mathfrak{i}_M\geq i$, $(M,F)$ is $\LGC(\rho)$, where
$\rho$ is the identity map of $[0,i)$. Corollary \ref{Balles} implies that
$\mu(B^+_p(\epsilon))\geq C(n,\delta) \epsilon^n$ for all $p\in M$ and
$\epsilon< i/(1+\sqrt{\delta})$. It follows from \cite[Proposition 3.11]{SZ1} that
\[
\Cov(M,\epsilon)\leq
\frac{\mu(M)}{C(n,\delta)(\epsilon/(2\sqrt{\delta}))^n}=C'(n,\delta,V)\epsilon^{-n}.
\]
Define the covering function
$N(\epsilon):=C'(n,\delta,V)\epsilon^{-n}$, $\epsilon\in (0,i/(1+\sqrt{\delta}))$.
The conclusion now follows from Lemma \ref{homotopfin}.
\end{proof}
One can easily see that Theorem \ref{Ya} implies Yamaguchi's finiteness theorem \cite{Y} and \cite[Theorem 1.3]{Z}.

\section{Appendix}

\begin{proposition}
Let $(M,F,d\mu)$ be a closed Finsler manifold. Then
\[
\mathbb{I}(M,d\mu)\leq \mathcal {S}(M,d\mu)\leq 2 \mathbb{I}(M,d\mu).
\]
\end{proposition}
\begin{proof}
Fix $\Gamma$ with $\mu(M_1)\leq \mu(M_2)$. Define a Lipschitz function $f^+_\epsilon$ by
\[
f^+_\epsilon(x):= \left \{
\begin{array}{lll}
&1, &x\in M_1,\,d(\Gamma,x)\geq \epsilon,\\
&\frac{1}{\epsilon}d(\Gamma,x), &x\in M_1,\,d(\Gamma,x)< \epsilon,\\
&0, &x\in M_2.
\end{array}
\right.
\]
By letting $\epsilon\rightarrow 0^+$, we obtain that
\begin{align*}
\inf_{\alpha\in \mathbb{R}}\left(\int_M|f_\epsilon^+-\alpha|^{\frac{n}{n-1}}d\mu \right)^{n-1}&\geq \inf_{\alpha\in \mathbb{R}}\left\{|1-\alpha|^{\frac{n}{n-1}}\mu(M_1)+|\alpha|^{\frac{n}{n-1}}\mu(M_2)\right\}^{n-1}\\
&\geq \mu(M_1)^{n-1}\inf_{\alpha\in \mathbb{R}}\left\{|1-\alpha|^{\frac{n}{n-1}}+ |\alpha|^{\frac{n}{n-1}}\right\}^{n-1}\\
&\geq \mu(M_1)^{n-1}/2.
\end{align*}
Set $\rho_+(x)=d(\Gamma,x)$, $x\in \overline{M_1}$. Lemma \ref{in-in} yields that $\nabla \rho_+|_\Gamma=\mathbf{n}_+$, where $\mathbf{n}_+$ denotes the unit inward normal vector field along $\partial M_1=\Gamma$. By the co-area formula (cf. \cite[Theorem 3.3.1]{Sh1}), we see that
\[
\int_MF^*(df_\epsilon^+)d\mu=\frac{1}{\epsilon}\int^\epsilon_0dt \int_{\rho_+^{-1}(t)}d\A_+\rightarrow \A_+(\Gamma).
\]
Hence, $2\A_+(\Gamma)^n\geq \mathcal {S}(M,d\mu)\cdot\mu(M_1)^{n-1}$. Similarly, define a Lipschitz function $f^-_\epsilon$ by
\[
f^-_\epsilon(x):= \left \{
\begin{array}{lll}
&0, &x\in M_2,\, d(\Gamma, x)>\epsilon\\
&\frac{1}{\epsilon}d(\Gamma,x)-1, &x\in M_2,\,d(\Gamma,x)\leq \epsilon,\\
&-1, &x\in M_1.
\end{array}
\right.
\]
Then one can show $2\A_-(\Gamma)^n\geq \mathcal {S}(M,d\mu)\cdot\mu(M_1)^{n-1}$.
Therefore, $ \mathcal {S}(M,d\mu)\leq 2\mathbb{I}(M,d\mu)$.

Given $f\in C^\infty$, let $\alpha_0$ be a median of $f$, i.e.,
\[
\mu(\{x: f(x)\geq \alpha_0\})\geq \frac12\mu(M), \ \mu(\{x: f(x)\leq \alpha_0\})\geq \frac12\mu(M).
\] Set $
M_1:=\{x:\, f(x)<\alpha_0 \}$ and $M_2:=\{x:\, f(x)>\alpha_0 \}$.
Then $\mu(M_i)\leq \mu(M)/2$, for $i=1,2$. Let $h:=f-\alpha_0$ and $h_i:=h|_{M_i}\in C^\infty_c(M_i)$, $i=1,2$.

Let $M_t:=\{x:\,h_2(x)>t\}$. Since $\mu(M_t)$ is decreasing, we have
\begin{align*}
\frac{d}{ds}\left(\int^s_0\mu(M_t)^{\frac{n-1}{n}}dt\right)^{\frac{n}{n-1}}&=\frac{n}{n-1}\left( \int^s_0\mu(M_t)^{\frac{n-1}{n}}dt \right)^{\frac1{n-1}}\mu(M_s)^{\frac{n-1}{n}}\\
&\geq \frac{n}{n-1}s^\frac1{n-1}\mu(M_s),
\end{align*}
which implies
\[
\left( \int^s_0\mu(M_t)^{\frac{n-1}{n}}dt\right)^{\frac{n}{n-1}}\geq\int^s_0\mu(M_t)dt^{\frac{n}{n-1}}.
\]
Note that $\nabla h_2$ is the inward normal vector field along $\partial M_t$. Thus,
\begin{align*}
&\int_{M_2} F^*(dh_2)d\mu=\int^\infty_0 \A_+(\partial M_t) dt\geq \mathbb{I}(M_2)^\frac1n\int^\infty_0\mu(M_t)^\frac{n-1}{n}dt\\
\geq& \mathbb{I}(M_2)^\frac1n\left(\int^\infty_0\mu(M_t)dt^{\frac{n}{n-1}} \right)^{\frac{n-1}{n}}=\mathbb{I}(M_2)^\frac1n\left(-\int^\infty_0t^{\frac{n}{n-1}}d\mu(M_t) \right)^{\frac{n-1}{n}}\\
=&\mathbb{I}(M_2)^\frac1n \left(\int^\infty_0 t^{\frac{n}{n-1}}dt\int_{\partial M_t}\frac{d\A_{\nabla h_2}}{F^*(dh_2)}  \right)^{\frac{n-1}{n}}=\mathbb{I}(M_2)^\frac1n\left(\int_M h_2^{\frac{n}{n-1}} d\mu\right)^{\frac{n-1}{n}}.
\end{align*}
Here, $\mathbb{I}(M_i)$ is defined by
\[
\inf_\Omega \frac{\min\{\A_\pm(\partial\Omega)\}^n}{\mu(\Omega)^{n-1}},
\]
where $\Omega$ range over all open submanifolds of $M_i$ with compact closures in $M_i$ and smooth boundary. Clearly, $\mathbb{I}(M_i)\neq 0$.

Likewise, one can show that $\int_{M_1} F^*(dh_1)d\mu\geq \mathbb{I}(M_1)^\frac1n\|h_1\|_{n/(n-1)}$. Since $\mu(M_i)\leq \mu(M)/2$,
$\mathbb{I}(M_i)\geq \mathbb{I}(M,d\mu)$.
Let $\chi_i$ be the characteristic function of $M_i$, $i=1,2$. Thus,
\begin{align*}
\int_M F^*(df)d\mu&=\int_M F^*(dh)d\mu=\sum_j \int_{M_j} F^*(dh_j)d\mu\\
&\geq \mathbb{I}(M,d\mu)^\frac1n\sum_j\left\{\int_M\chi_j|f-\alpha_0|^{\frac{n}{n-1}} \right\}^{\frac{n-1}{n}}\\
&\geq \mathbb{I}(M,d\mu)^\frac1n\|f-\alpha_0\|_{\frac{n}{n-1}}\geq \mathbb{I}(M,d\mu)^\frac1n \inf_{\alpha\in \mathbb{R}}\|f-\alpha\|_{\frac{n}{n-1}},
\end{align*}
which implies that $\mathcal {S}(M,d\mu)\geq \mathbb{I}(M,d\mu)$.\end{proof}


\begin{thebibliography}{10}

\bibitem{AlB} J. Alvarez-Paiva and G. Berck, \textsl{What is wrong with the
Hausdorff measure in Finsler spaces}, Adv. in Math.,
{\bf{204}}(2006), 647-663.

\bibitem{AlT}J. Alvarez-Paiva and A.C. Thompson, \textsl{ Volumes in normed and Finsler
spaces}, A Sampler of Riemann-Finsler geometry (Cambridge) (D. Bao,
R. Bryant, S.S. Chern, and Z. Shen, eds.), Cambridge University
Press, 2004, pp. 1-49.




\bibitem{BCS} D. Bao, S. S. Chern and Z. Shen, \textsl{An introduction
to Riemannian-Finsler geometry}, GTM {\bf{200}}, Springer-Verlag,
2000.




\bibitem{BBI} D. Burago, Y. Burago and S. Ivanov, \textsl{A course in metric geometry}, A.M.S, 2001.


\bibitem{C2} I. Chavel, \textsl{Eigenvalues in Riemannian geometry}, Academic Press, New York, 1984.






\bibitem{C3} B. Chen, \textsl{Some geometric and analysis problems in Finsler geometry}, Doctoral thesis, Zhejiang University, 2010.

\bibitem{Cr2} C. Croke, \textsl{A sharp four dimensional isoperimetric inequality}, Comment. Math. Helv. \textbf{59}(1984),
187-192.


\bibitem{Cr3} C. Croke, \textsl{Curvature Free Volume Estimates}, Inventiones Mathematicae,
\textbf{76}(1984), 515-521.

\bibitem{Cr5} C. Croke, N. Dairbekov, \textsl{Lengths and volumes in Riemannian manifolds},
Duke Math. J., \textbf{125 }(2004), 1-14.

\bibitem{Cr} C. Croke,  \textsl{Some  isoperimetric  inequalities  and  eigenvalue  estimates},  Ann.  Sci.  Ec.
Norm.  Super,  Ser.  \textbf{13}(1980), 419-435.


\bibitem{Cr4} C. Croke, and M. Katz,\textsl{Universal volume bounds in Riemannian manifolds}, Surveys in Differential Geometry VIII, Lectures on Geometry and
Topology held in honor of Calabi, Lawson, Siu, and Uhlenbeck at Harvard University, May 3-5, 2002, edited by S.T. Yau (Somerville, MA:
International Press, 2003.) pp. 109-137.


\bibitem{E} D. Egloff, \textsl{Uniform Finsler Hadamard manifolds}, Ann. Inst. Henri Poincar\'e, \textbf{66}(1997), 323-357.

\bibitem{L} M. Ledoux, \textsl{A simple analytic proof of an inequality by P. Buser}, Proc. Amer.
Math. Soc. \textbf{121}(1994), 951-959.




\bibitem{GS} Y. Ge and Z. Shen, \textsl{Eigenvalues and eigenfunctions of metric measure manifolds}, Proc.
London Math. Soc., \textbf{82}(2001), 725-746.

\bibitem{R} H. Rademacher, \textsl{Nonreversible Finsler metrics of positive
ag curvature},  A sampler
of Riemann-Finsler geometry, Cambridge Univ. Press, Cambridge, 2004, pp. 261-302.



\bibitem{Sa} L. Santal\'o, \textsl{Integral Geometry and Geometric Probability}, Encyclopedia Math. Appl.,
Add- ison-Wesley, Reading, MA, 1976.

\bibitem{Sa2} L. Santal\'o, \textsl{Measure of sets of geodesics in a Riemannian space and applications to integral formulas in elliptic and hyperbolic spaces}, Summa Brasil. Math., \textbf{3}(1952), 1-11.






\bibitem{Sh1} Z. Shen, \textsl{Lectures on Finsler geometry}, World
Sci., Singapore, 2001.


\bibitem{Sh3} Z. Shen, \textsl{The non-linear Laplacian for Finsler manifolds},  The theory of Finslerian
Laplacians and applications, vol. 459 of Math. Appl., Kluwer Acad. Publ., Dordrecht,
1998, pp. 187-198.




\bibitem{SZ1} Y. Shen and W. Zhao, \textsl{Gromov pre-compactness theorems for nonreversible Finsler manifolds}, Diff. Geom. Appl., \textbf{28}(2010),
565-581.






\bibitem{W} B. Wu, \textsl{Volume form and its applications in Finsler geometry}, Publ. Math. Debrecen, \textbf{78}(2011),
723-741.

\bibitem{Y} T. Yamaguchi, \textsl{Homotopy type finiteness theorems for certain precompact families of
Riemannian manifolds}, Proc. Am. Math. Soc. \textbf{102}(1988), 660-666.

\bibitem{ZY} W. Zhao and Y. Shen, \textsl{A Universal Volume Comparison Theorem for Finsler Manifolds and Related Results}, Can. J. Math., \textbf{65}(2013), 1401-1435.

\bibitem{Z} W. Zhao, \textsl{Homotopy finiteness theorems for Finsler manifolds}, Publ. Math. Debrecen, \textbf{83}(2013), 329-358.


\end{thebibliography}
\end{document}